\documentclass[10pt,a4paper,reqno]{article}

\usepackage[T1]{fontenc}
\usepackage[all]{xy}
\usepackage{epsfig}
\usepackage{amssymb}
\usepackage{amsmath}
\usepackage{graphics,graphicx}
\usepackage{eepic}
\usepackage{epsfig}
\usepackage{trees}
\usepackage{vmargin}
\usepackage{theorem}
\usepackage{amsbsy}
\usepackage{amsfonts}
\usepackage{indentfirst}
\usepackage{enumerate}

\newtheorem{theorem}{Theorem}[section]
\newtheorem{lemma}[theorem]{Lemma}
\newtheorem{prop}[theorem]{Proposition}

\newtheorem{remark}[theorem]{Remark}

\newtheorem{example}[theorem]{Example}

\newenvironment{proof}{\noindent\text{\textbf{Proof.\:}}}{}

\def\qed{\hfill $\square$ \goodbreak}

\def\eps{\varepsilon}
\def\R{{\mathbb R}}
\def\haus{\mathcal{H}^{d-1}}
\def\F{\mathcal{F}}

\def\Ind{\raise 2pt\hbox{$\chi $}}

\newbox\boxint\newbox\boxtir\newdimen\longa\newdimen\longb
\def\intgm{
\setbox\boxint=\hbox{$\displaystyle\int$}
\setbox\boxtir=\hbox{$-$}
\longa=\wd\boxint\longb=\wd\boxtir
\advance\longa-\longb\divide\longa2
\advance\longb\longa
\box\boxint\hskip-\longb\box\boxtir}
\def\intmoy{\mathop{\intgm}\nolimits}
\def\moyg_#1{\intmoy_{\!\!\!#1}}
\def\intm{
\setbox\boxint=\hbox{$\int$}
\setbox\boxtir=\hbox{{--}}
\longa=\wd\boxint\longb=\wd\boxtir
\advance\longa-\longb\divide\longa2
\advance\longb\longa
\box\boxint\hskip-\longb\box\boxtir}
\def\moyp{\mathop{\intm}\nolimits}
\def\moy_#1{\ifmmode\ifinner{\moyp_{#1}}\else{\moyg_{#1}}\fi\fi}

\title{Regularity of the optimal shape for the first
eigenvalue of the Laplacian
with volume and inclusion constraints}
\author{Tanguy Brian\c{c}on\footnote{Lyc\'ee Agora, 92800 Puteaux , France,
briancon\_tanguy@yahoo.fr},
Jimmy Lamboley\footnote{ENS Cachan Bretagne, IRMAR, UEB, av Robert Schuman, 35170 Bruz, France,
jimmy.lamboley@bretagne.ens-cachan.fr}}
\begin{document}

\maketitle
\begin{abstract}
We consider the well-known following shape optimization problem:
$$\lambda_1(\Omega^*)=\min_{\stackrel{|\Omega|=a}
{\Omega\subset{ D}}} \lambda_1(\Omega),
$$ where $\lambda_1$ denotes the first eigenvalue of the Laplace operator with homogeneous Dirichlet boundary condition, and $D$ is an open bounded set (a box).
 It is well-known that the solution of this problem is the ball of volume $a$ if such a ball exists in the box $D$ (Faber-Krahn's theorem).\\
In this paper, we prove regularity properties of the boundary of the optimal shapes $\Omega^*$ in any case and in any dimension.
Full regularity is obtained in dimension 2.\\

\smallskip

\end{abstract}

\begin{section}{Introduction and main results}
Let $D$ be a bounded open subset of $\R^d$. For all open subset $\Omega$ of $D$, we denote by $\lambda_1(\Omega)$ the first eigenvalue of the Laplace
operator in $\Omega$, with homogeneous boundary conditions, and by $u_\Omega$ a normalized eigenfunction, that is
$$\left\{
\begin{array}{ccll}
-\Delta u_{\Omega}&=&\lambda_1(\Omega)u_{\Omega}&in\;\;\Omega,\\
u_{\Omega}&=&0&on\;\;\partial\Omega,\\
\int_\Omega u_\Omega^2&=&1.
\end{array}
\right.
$$
We are interested here in the regularity of the optimal shapes of the following shape optimization problem, where $a\in(0,|D|)$ ($|D|$ denotes
the Lebesgue measure of $D$):
\begin{equation}\label{probforme}
\left\{
\begin{array}{ll}
\displaystyle{\Omega^*\textrm{ open, }\Omega^*\subset D ,|\Omega^*|=a,}\\[2mm]
\displaystyle{\lambda_1(\Omega^*)=\min\{\lambda_1(\Omega);\;\Omega^*\textrm{ open, }\Omega\subset D,|\Omega|=a\}.}
\end{array}
\right.
\end{equation}
By a well-known theorem of Faber and Krahn, if there is a ball $B\subset D$ with $|B|=a$, then this ball is an optimal shape and it is unique, up to
translations (and up to sets of zero capacity).\\

Here we adress the question of existence of a \textbf{regular} optimal set in all cases.\\

Existence of a \textit{quasi-open} optimal set $\Omega^*$ may be deduced from a general existence result by G. Buttazzo and G. Dal Maso (see \cite{BdM})
for an extended version of (\ref{probforme}), where the variable sets $\Omega$ are not necessarily open. An optimal shape
$\Omega^*$  may not be more than a quasi-open set if $D$ is not connected (we reproduce in the appendix the example mentioned in \cite{BHP}).
On the other hand, it is proved in \cite{BHP} or \cite{H00} that such an open optimal set $\Omega^*$ always exists for (\ref{probforme}) and, if moreover
$D$ is connected, then all optimal shapes $\Omega^*$ are open. More precisely, it is proved in \cite{BHP} that,
for any $D$, $u_{\Omega^*}$ is locally
Lipschitz continuous in $D$. If moreover $D$ is connected, then $\Omega^*$ coincides with the support of $u_{\Omega^*}$ (and is therefore open).
Let us summarize this as follows (see also \cite{Wa}):
\begin{prop}\label{lip}
Assume $D$ is open and bounded. The problem (\ref{probforme}) has a solution $\Omega^*$, and $u_{\Omega^*}$ is
nonnegative and
 locally Lipschitz continuous in $D$. If $D$ is connected, $\Omega^*=\{x\in D, u_{\Omega^*}>0\}$.
\\Moreover, we have
\begin{equation}\label{deltaumes}
\Delta u_{\Omega^*} + \lambda_1(\Omega^*) u_{\Omega^*}\geq 0\;in\;D,
\end{equation}
which means that $\Delta u_{\Omega^*} + \lambda_1(\Omega^*) u_{\Omega^*}$ is a positive
Radon measure.
\end{prop}
Here, we are interested in the regularity of $\partial\Omega^*$ itself, and we prove the following theorem:
\begin{theorem}\label{theoreg}
Assume $D$ is open, bounded and connected. Then any solution of (\ref{probforme}) satisfies:
\begin{enumerate}
\item $\Omega^*$ has locally finite perimeter in $D$ and
\begin{equation}\label{haus}
\haus((\partial\Omega^*\setminus\partial^*\Omega^*)\cap D)=0,
\end{equation}
where $\haus$ is the Hausdorff measure of dimension $d-1$, and
$\partial^*\Omega^*$ is the reduced boundary (in the sense of sets
with finite perimeter, see \cite{EG} or \cite{G}).
\item There exists $\Lambda>0$ such that
\[\Delta u_{\Omega^*}+\lambda_1(\Omega^*)u_{\Omega^*}
=\sqrt{\Lambda} \haus\lfloor\partial\Omega^*, \]
in the sense of distribution in $D$, where $\haus\lfloor\partial\Omega^*$ is
the restriction of the $(d-1)$-Hausdorff measure to $\partial\Omega^*$.
\item $\partial^*\Omega^*$ is an analytic hypersurface in $D$.
\item If $d=2$, then the whole boundary $\partial\Omega^*\cap D$ is
analytic.
\end{enumerate}
\end{theorem}
We use the same strategy as in \cite{B} (where the regularity is studied for another
shape optimization problem). Theorem \ref{theoreg} essentially relies on the proof of the equivalence of (\ref{probforme})
with a penalized version for the constraint $|\Omega|=a$, as stated in Theorem \ref{theopena} below. Once we have this penalized version,
we can use techniques and results from \cite{AC} (see also \cite{GSh} and \cite{B}).
\begin{remark}\label{remCaf}
According to the results in \cite{AC}, the third point in Theorem \ref{theoreg} is a direct consequence of the second one
 which says that $u_{\Omega^*}$ is a ``\textup{weak solution}'' in the sense of \cite{AC}.
 To obtain the full regularity of the boundary for $d=2$, the fact that $u_{\Omega^*}$ is
a weak solution is not sufficient, and more information has to be deduced from the variational problem. The approach is essentially the same
as in Theorem 6.6 and Corollary 6.7 in \cite{AC}. The necessary adjustments are given at the end of this paper.
\end{remark}
\begin{remark}
According to the result of \cite{W2,W,CJK,DSJ}, it is likely that full regularity of the boundary may be extended to higher dimension ($d\leq 6$ ?),
and therefore that the estimate $(\ref{haus})$ can be improved.\\
But this needs quite more work and is under study.
\end{remark}

By a classical variational principle, we know that, for all $\Omega\subset D$ open,
\begin{equation}\label{defuomega}
\lambda_1(\Omega)=\int_{\Omega} |\nabla u_{\Omega}|^2
=\min\left\{\int_{\Omega}|\nabla u|^2, u\in H^1_0(\Omega), \int_{\Omega}u^2=1 \right\}.
\end{equation}
Here, $\lambda_1(\Omega^*)\leq\lambda_1(\Omega)$ for all open set $\Omega\subset D$ with $|\Omega|=a$.
Since $\left[\Omega\subset\widetilde{\Omega}\;\Rightarrow \lambda_1(\Omega)\geq\lambda_1(\widetilde{\Omega})\right]$, it follows that
$\lambda_1(\Omega^*)\leq\lambda_1(\Omega)$ for all open set $\Omega\subset D$ with $|\Omega|\leq a$. Coupled with (\ref{defuomega}), this leads to the
following variation property of $\Omega^*$ and $u_{\Omega^*}$ (see \cite{BHP} for more details), where
 we denote $u=u_{\Omega^*}, \lambda_a=\lambda_1(\Omega^*)$, and $\Omega_v=\{x\in D; v(x)\neq 0\}$:
\begin{equation}\label{probfonc}
\lambda_a=\int_{D}|\nabla u|^2=\min\left\{\int_{D}|\nabla v|^2;v\in H^1_0(D),
\int_D v^2=1, |\Omega_v|\le a \right\}.
\end{equation}
Let us rewrite this as follows. For $w\in H^1_0(D)$, we denote
$\displaystyle{J(w)=\int_D|\nabla w|^2-\lambda_a\int_D w^2}$. Then applying (\ref{probfonc}) with $v=w/(\int_D w^2)^{1/2}$, we obtain that $u$ is
a solution of the following optimization problem:
\begin{equation}\label{probfonc2}
J(u)\le J(w), \textrm{ for all }w\in H^1_0(D),\textrm{ with }|\Omega_w|\le a.
\end{equation}
One of the main ingredient in the proof of Theorem \ref{theoreg} is to improve the variational property (\ref{probfonc2}) in two directions, as stated
 in Theorem \ref{theopena} below. The approach is local.\\

Let $B_R$ be a ball
included in $D$ and centered on $\partial\Omega_u\cap D$. We define
\[\F=\{v\in H^1_0(D), u-v\in H^1_0(B_R)\}. \]
For $h>0$, we denote by $\mu_-(h)$ the biggest $\mu_-\geq 0$ such that,
\begin{equation}\label{mu}
\forall\; v \in\F\textrm{ such that } a-h\le |\Omega_v|\le a,\;
J(u)+\mu_-|\Omega_u|\le J(v)+\mu_-|\Omega_v|.
\end{equation}
We also define $\mu_+(h)$ as the smallest $\mu_+\geq 0$ such that,
\begin{equation}\label{lambda}
\forall\; v \in\F\textrm{ such that } a\le |\Omega_v|\le a+h,\;
J(u)+\mu_+|\Omega_u|\le J(v)+\mu_+|\Omega_v|.
\end{equation}
The following theorem is a main step in the proof of Theorem \ref{theoreg}:
\begin{theorem}\label{theopena}
Let $u,B_R$ and  $\F$ as above. Then for $R$ small enough
(depending only on $u,a$ and $D$), there exists $\Lambda>0$ and $h_0>0$ such that,
\[\forall\; h\in(0,h_0),\; 0<\mu_-(h)\le\Lambda\le\mu_+(h)<+\infty, \]
and, moreover,
\begin{equation}\label{limmu}
\lim_{h\to 0}\mu_+(h)=\lim_{h\to 0}\mu_-(h)=\Lambda.
\end{equation}
\end{theorem}

\begin{remark} We can compare the existence of $\mu_+(h)$ with Theorem 2.9 in
\cite{BHP}. This theorem shows that there exists $\mu_+$ such that
\[\int_D |\nabla u|^2\le \int_D|\nabla v|^2+
\lambda_a\left[1-\int_D v^2\right]^+ + \mu_+(|\Omega_v|-a), \]
for $v\in H^1_0(D)$ and $|\Omega_v|\geq a$. The difference with \cite{BHP}
is that, in $(\ref{lambda})$, we have the term $\lambda_a[1-\int_D v^2]$ (not
only the positive part), but we allowed only perturbations in $B_R$.
We cannot expect to have something like $(\ref{lambda})$
for perturbations in all $D$ (because we may find $v$ with
$|\Omega_v|>a$ and $J(v)<0$, so $lim_{t\to+\infty} J(tv)=-\infty$).
\end{remark}

In the next section, we will prove Theorem \ref{theopena}. In the third
section, we will prove Theorem \ref{theoreg}. In the appendix, we discuss the case $D$ non-connected.
\end{section}

\begin{section}{Proof of Theorem \ref{theopena}}
In the next lemma, we give an Euler-Lagrange equation
for our problem. The proof follows the steps of the
Euler-Lagrange equation in $\cite{C}$.
\begin{lemma}[Euler-Lagrange equation]\label{euler}
Let $u$ be a solution of $(\ref{probfonc2})$. Then there exists
$\Lambda\geq 0$ such that, for all $\Phi\in C^{\infty}_0(D,\R^d)$,
\[\int_D 2(D\Phi\nabla u,\nabla u)-\int_D  |\nabla u|^2\nabla\cdot\Phi+\lambda_a
\int_D u^2 \nabla\cdot \Phi = \Lambda\int_{\Omega_u}\nabla\cdot \Phi\;. \]
\end{lemma}
\begin{proof} We start by a general remark that will be useful in the
rest of the paper. If $v\in H^1_0(D)$ and if $\Phi\in C^{\infty}_0(D,\R^d)$,
we define $v_t(x)=v(x+t\Phi(x))$; therefore, for $t$ small enough, $v_t\in H^1_0(D)$.
A simple calculus gives (when $t$ goes to 0),
\[|\Omega_{v_t}|=|\Omega_v|-t\int_{\Omega_v}\nabla\cdot \Phi + o(t), \]
\[J(v_t)=J(v)+t\left(
\int_D 2(D\Phi\nabla v.\nabla v)-\int_D  |\nabla v|^2\nabla\cdot \Phi+\lambda_a
\int_D v^2 \nabla\cdot \Phi \right) + o(t). \]
Now we apply this with $v=u$ and $\Phi$ such that $\int_{\Omega_u}\nabla\cdot \Phi>0$.
Such a $\Phi$ exists, otherwise we would get, using that $D$ is connected,
$\Omega_u=D$ or $\emptyset$ a.e. We have $|\Omega_{u_t}|<|\Omega_u|$
for $t\geq 0$ small enough and, by minimality,
\begin{eqnarray*}
J(u) & \le  & J(u_t) \\
& = & J(u) +t\left(
\int_D 2(D\Phi\nabla u,\nabla u)-\int_D  |\nabla u|^2\nabla\cdot \Phi+\lambda_a
\int_D u^2 \nabla\cdot \Phi \right) + o(t),
\end{eqnarray*}
and so,
\begin{equation}\label{eulerp}
\int_D 2(D\Phi\nabla u,\nabla u)-\int_D  |\nabla u|^2\nabla\cdot \Phi+\lambda_a
\int_D u^2 \nabla\cdot \Phi \geq 0.
\end{equation}
Now, we take $\Phi$ with $\int_{\Omega_u}\nabla\cdot \Phi =0$.
Let $\Phi_1$ be such that  $\int_{\Omega_u}\nabla\cdot \Phi_1=1$. Writing
$(\ref{eulerp})$ with $\Phi+\eta\Phi_1$ and letting $\eta$ goes to $0$, we get
$(\ref{eulerp})$ with this $\Phi$ and, using $-\Phi$, we get $(\ref{eulerp})$
with an equality instead of the inequality. For a general $\Phi$, we use this
equality with $\Phi-\Phi_1(\int_{\Omega_u}\nabla\cdot \Phi)$ (we have
$\displaystyle{\int_{\Omega_u}\nabla\cdot \left(\Phi-\Phi_1\left(\int_{\Omega_u}\nabla\cdot  \Phi\right)\right)=0}$), and we get the
result with
\[\Lambda=\int_D 2(D\Phi_1\nabla u,\nabla u)-\int_D |\nabla u|^2\nabla\cdot \Phi_1 
+\lambda_a\int_D u^2 \nabla\cdot \Phi_1\geq 0,\]
using $(\ref{eulerp})$.\qed
\end{proof}
\begin{remark}We will have to prove that, in fact, $\Lambda>0$.
\end{remark}

Let us remind our notations: let $u$ be a solution of (\ref{probfonc2}), and let $B_R$ be a ball
included in $D$ and centered on $\partial\Omega_u\cap D$. We define
\[\F=\{v\in H^1_0(D), u-v\in H^1_0(B_R)\}. \]
Before proving Theorem \ref{theopena}, we give the following useful lemma:
\begin{lemma}
Let $u,B_R$ and $\mathcal{F}$ as above. Then there exists a constant $C$ such that,
for $R$ small enough,
\[\forall v\in\F,\;\; J(v)\geq \frac{1}{2}\int_{B_R}|\nabla v|^2 -C. \]
\end{lemma}
\begin{proof} We know that $\lambda_1(B_R)=\lambda_1(B_1)/(R^2)$ (we just
use the change of variable $x\to x/R$). If $R$ is small enough we have:
\begin{equation}\label{choixr}
\lambda_1(B_R)\geq 1,\ \frac{4\lambda_a}{\lambda_1(B_R)}\le 1/2.
\end{equation}
Let $v\in\F$; so $u-v\in H^1_0(B_R)$, and using the variational formulation of
$\lambda_1(B_R)$, we get
\[\|u-v\|^2_{L^2(B_R)}\le\frac{\|\nabla(u-v)\|^2_{L^2(B_R)}}{\lambda_1(B_R)}. \]
We deduce that,
\begin{eqnarray*}
\|v\|^2_{L^2(B_R)} &\le &
2\frac{\|\nabla (u-v)\|^2_{L^2(B_R)}}{\lambda_1(B_R)}+2\|u\|^2_{L^2(B_R)} \\
& \le & 4\frac{\|\nabla v\|^2_{L^2(B_R)}}{\lambda_1(B_R)} + \frac{C}{\lambda_a},
\end{eqnarray*}
(we use $(\ref{choixr})$) where $C$ depends only on the $L^2$ norms of $u$
and his gradient. Now we have
\[J(v) \geq \int_D |\nabla v|^2-
\lambda_a\left(4\frac{\|\nabla v\|^2_{L^2(B_R)}}{\lambda_1(B_R)} + \frac{C}{\lambda_a}\right), \]
and we get the result using $(\ref{choixr})$.\qed
\end{proof}
\begin{remark}\label{jbound} This lemma is interesting for two reasons.
The first one is that $J$ is bounded from below on $\F$.
The second one is that, if $v_n\in\F$ is a sequence
such that $J(v_n)$ is bounded,
then $\|\nabla v_n\|_{L^2(B_R)}$ is also bounded.
Since $v_n=u$ outside $B_R$ we deduce that $v_n$ is bounded in $H^1_0(D)$
(and so weakly converges up to a sub-sequence...).
\end{remark}
\textbf{Proof of Theorem \ref{theopena}: }We divide our proof into four parts. Let $\Lambda\geq0$ be as in
Lemma \ref{euler}.\\
\textbf{ First part: $\Lambda\le\mu_+(h)<+\infty$.}

We start the proof by showing that $\mu_+(h)$ is finite. Since $B_R$
is centered on the boundary on $\partial\Omega_u$, we first show:
\[0<|\Omega_u\cap B_R|<|B_R|. \]
The first inequality comes from the fact that $\Omega_u$ is open. The second one comes from the following lemma:
\begin{lemma}\label{mes}
Let $\omega$ be an open subset of $D$, and let $u$ be a solution of $(\ref{probfonc2})$. If $|\Omega_u\cap \omega|=|\omega|$, then
$$-\Delta u = \lambda_a u\;\;in\;\;\omega,$$
and therefore $\omega\subset\Omega_u$.
\end{lemma}
\textbf{Proof of Lemma \ref{mes}.}
Since $u>0$ a.e. on $\omega$, we define $v\in H^1_0(D)$ by
$v=u$ outside $\omega$ and  $-\Delta v=\lambda_a u$ in $\omega$.
From the strict maximum principle, we get $v>0$ on $\omega$ and $|\Omega_v|=|\Omega_u|$.
By minimality ($J(u)\le J(v)$) we have,
\begin{eqnarray*}
\int_{\omega}(\nabla u - \nabla v).(\nabla u -\nabla v + 2 \nabla v)
- \lambda_a\int_{\omega}(u-v)(u+v) & \le & 0 \\
\int_{\omega}|\nabla u - \nabla v|^2+
\lambda_a\int_{\omega}(u-v)(2u-u-v) & \le & 0,
\end{eqnarray*}
(we use that $u-v\in H^1_0(\omega)$ and $-\Delta v=\lambda_a u$ in $\omega$).
We get that $u=v$ a.e. in $\omega$ and by continuity $u=v>0$ everywhere
in $\omega$.
\qed
\bigskip
If $|\Omega_u\cap B_R|=|B_R|$, applying this lemma to $\omega=B_R$, we would get $\Omega_u\cap B_R=B_R$,
which is impossible since $B_R$ is centered on $\partial\Omega_u$. If $R$ is small enough we can also suppose,
\[0<|\Omega_u\setminus B_R|<|D\setminus B_R|. \]
For the first inequality, we need that $|B_R|<a$, and for the second one
we need $a<|D|-|B_R|$.

Let $h>0$ be such that $h<|B_R|-|\Omega_u\cap B_R|$ (and so, if
$v\in\F$ with $|\Omega_v|\le a+h$, then $|\Omega_v\cap B_R|<|B_R|$).
Let $(\mu_n)$
an increasing sequence to $+\infty$. There exists $v_n\in\F$
such that $|\Omega_{v_n}|\leq a+h$ and,
\begin{equation}\label{vnla}
J(v_n)+\mu_n(|\Omega_{v_n}|-a)^+= \min_{\ v\in\F, |\Omega_{v}|\le a+h} \left\{J(v)+\mu_n(|\Omega_v|-a)^+\right\}.
\end{equation}
For this we use remark \ref{jbound}, and so the functional
$J(v)+\mu_n(|\Omega_v|-a)^+$ is bounded by below for
$v\in\F$. Moreover, a minimizing sequence for this functional
is bounded in $H^1_0(D)$ and so weakly converges in $H^1_0(D)$,
strongly in $L^2(D)$ and almost everywhere (up to a sub-sequence)
to some $v_n$. Using the lower semi-continuity of
$v\to \int_D |\nabla v|^2$ for the weak convergence, the strong convergence
in $L^2(D)$ and the lower semi-continuity of $v\to |\Omega_v|$ for
the convergence almost everywhere we see that $v_n$ is such that
$(\ref{vnla})$ is true.

If $|\Omega_{v_n}|\le a$ then $(\ref{lambda})$ is true
with $\mu_n$, so we will suppose to the contrary that $|\Omega_{v_n}|>a$ for all $n$. \\

\textbf{Step 1: Euler-Lagrange equation for $v_n$.}
If we set $b_n=|\Omega_{v_n}|$, then  $v_n$ is also solution of
\[J(v_n) =\min_{v\in\F, |\Omega_v|\leq b_n} J(v). \]
With the same proof as in lemma \ref{euler}, we can write an
Euler-Lagrange equation for $v_n$ in $B_R$. That is, there exists
$\Lambda_n\geq 0$ such that, for $\Phi\in C^{\infty}_0(B_R,\R^d)$,
\begin{equation}\label{eulern}
\int_D 2(D\Phi\nabla v_n.\nabla v_n)-\int_D  |\nabla v_n|^2 \nabla\cdot \Phi+\lambda_a
\int_D v_n^2 \nabla\cdot \Phi = \Lambda_n\int_{\Omega_{v_n}}\nabla\cdot \Phi.
\end{equation}

\textbf{Step 2: $\Lambda_n\geq \mu_n$.} There exists $\Phi\in C^{\infty}_0(B_R)$ such that $\int_{\Omega_{v_n}} \nabla\cdot  \Phi=1$. Let $v_n^t(x)=v_n(x+t\Phi(x))$.
We have $v_n^t\in\mathcal{F}$ for $t\geq 0$ small enough,
and using derivation results recalled in the proof of lemma \ref{euler}
and $|\Omega_{v_n}|>a$, we get
\[a<|\Omega_{v^t_n}|=|\Omega_{v_n}|-t+o(t)\le a+h, \]
\[J(v_n^t)=J(v_n)+ t \Lambda_n  + o(t). \]
Now we use $(\ref{vnla})$ with $v=v_n^t$ in order to get,
\[J(v_n)+\mu_n(|\Omega_{v_n}|-a)\le J(v_n)+ t \Lambda_n  + o(t)
+\mu_n(|\Omega_{v_n}|-t-a), \]
and dividing by $t>0$ and letting $t$ goes to 0, we finally get
$\Lambda_n\geq \mu_n$. \\

\textbf{Step 3: $v_n$ strongly converges to some $v$.} Using (\ref{vnla}) with $v=u$, we get
\begin{equation}\label{step3u}
J(v_n)+\mu_n(|\Omega_{v_n}|-a)\le J(u)
\end{equation}
and so, using Remark \ref{jbound},
we can deduce that $v_n$ weakly converge in $H^1_0$ (up to a sub-sequence)
to some $v\in\F$ with $|\Omega_v|\le a+h$. We also have the strong
convergence in $L^2(D)$ and the convergence almost everywhere. Since $J$
is bounded from below on $\F$, we see from $(\ref{step3u})$
that $\mu_n(|\Omega_{v_n}|-a)$ is bounded and we get
$\lim_{n\to\infty}|\Omega_{v_n}|=a$, and so $|\Omega_v|\le a$.
From $J(v_n)\le J(u)$, we get $J(v)\le\liminf J(v_n)\le J(u)$
and so $v$ is a solution of (\ref{probfonc2}). Finally we can write, using (\ref{vnla}), that
$J(v_n)\le J(v)$ and we get, using the strong convergence of $v_n$ in $L^2$,
\[\limsup_{n\to\infty}\int_{D}|\nabla v_n|^2\le \int_{D}|\nabla v|^2. \]
We also have, with weak convergence in $H^1_0(D)$ that
\[\int_{D}|\nabla v|^2\le\liminf_{n\to\infty}\int_{D}|\nabla v_n|^2. \]
We deduce that $\lim_{n\to\infty}\|\nabla v_n\|_{L^2(D)}=
\|\nabla v\|_{L^2(D)}$. With the weak-convergence, this gives the strong
convergence of $v_n$ to $v$ in $H^1_0(D)$. \\

\textbf{Step 4: $\lim\Lambda_n=\Lambda$.} We see that $v$ is a solution of $(\ref{probfonc2})$,
so we can apply Lemma \ref{euler} to get that there exists a
$\Lambda_v$ such that
\[\forall\;\Phi\in C^{\infty}_0(D,\R^d),\;\int_D 2(D\Phi\nabla v.\nabla v)-\int_D  |\nabla v|^2\nabla\cdot \Phi+\lambda_a
\int_D v^2 \nabla\cdot  \Phi = \Lambda_v\int_{\Omega_v}\nabla\cdot \Phi. \]
We have $u=v$ outside $B_R$ so, using this
equation and the Euler-Lagrange equation for $u$ we see that
$\Lambda_v=\Lambda$. Now, we write the Euler-Lagrange for $v_n$
and $\Phi\in C^{\infty}_0(D,\R^d)$ such that $\int_{\Omega_v} \nabla\cdot \Phi\neq 0$,
\[\int_D 2(D\Phi\nabla v_n.\nabla v_n)-\int_D  |\nabla v_n|^2\nabla\cdot \Phi+\lambda_a
\int_D v_n^2 \nabla\cdot  \Phi = \Lambda_n\int_{\Omega_{v_n}}\nabla\cdot \Phi, \]
and, using the strong convergence of $v_n$ to $v$, we get that
\begin{eqnarray*}
\lim_{n\to\infty}\Lambda_n & = &
\lim_{n\to\infty} \frac{\int_D 2(D\Phi\nabla v_n.\nabla v_n)-
\int_D |\nabla v_n|^2\nabla\cdot \Phi +\lambda_a
\int_D v_n^2 \nabla\cdot  \Phi}{\int_{\Omega_{v_n}}\nabla\cdot \Phi} \\
& = &
\frac{\int_D 2(D\Phi\nabla v.\nabla v)-\int_D  |\nabla v|^2\nabla\cdot \Phi+\lambda_a
\int_D v^2 \nabla\cdot  \Phi}{\int_{\Omega_{v}}\nabla\cdot \Phi} \\
& =& \Lambda.
\end{eqnarray*}

Since $\lim\mu_n=+\infty$ we get the contradiction from Steps 2 and 4, and so $\mu_+(h)$ is finite.

To conclude this first part, we now have to see that $\Lambda\le\mu_+(h)$. Let $\Phi\in C^{\infty}_0$  be such
that $\int_{\Omega_u}\nabla\cdot \Phi=-1$, and let $u_t(x)=u(x+t\Phi(x))$.
Using the calculus in the proof of Lemma \ref{euler} we have,
for $t\geq 0$ small enough,
\[a\le |\Omega_{u_t}|=a+t+o(t)\le a+h, \]
\[J(u_t)=J(u)-t\Lambda +o(t). \]
Now, using $(\ref{lambda})$, we have
\[J(u)+\mu_+(h)a\le J(u)-t\Lambda +\mu_+(h)(a+t)+o(t), \]
and we get $\Lambda\le\mu_+(h)$. \\
~\\
\textbf{Second part: $\lim\mu_+(h)=\Lambda$. }

We first see that $\mu_+(h)>0$ for $h>0$. Indeed, if $\mu_+(h)=0$ we write
\[\textrm{for every }\varphi\in C^{\infty}_0(B_R)\textrm{ with }|\{\varphi\neq 0\}|<h,\; J(u)\le J(u+t\varphi), \]
so
\[-\Delta u=\lambda_a u \textrm{ in } B_R, \]
which contradicts $0< |\Omega_u\cap B_R|<|B_R|.$

Let $\eps>0$ and $h_n>0$ a decreasing sequence tending to 0.
Because $h\to\mu_+(h)$ is non-increasing, we just have to
see that $\lim\mu_+(h_n)\le\Lambda+\eps$ for a sub-sequence of $h_n$.
If $\Lambda>0$, let $\eps\in]0,\Lambda[$ and
$0<\alpha_n:=\mu_+(h_n)-\eps<\mu_+(h_n)$; if $\Lambda=0$, let
$0<\alpha_n=\mu_+(h_n)/2<\mu_+(h_n)$.
There exists $v_n$ such that
\begin{equation*}
J(v_n)+\alpha_n(|\Omega_{v_n}|-a)^+=\min_{v\in\F, |\Omega_v|\le a+h_n}\left\{J(v)+\alpha_n(|\Omega_{v}|-a)^+\right\}.
\end{equation*}
Since $\alpha_n< \mu_+(h_n)$ we see that $|\Omega_{v_n}|>a$
(otherwise we write $J(u)\le J(v_n)+\alpha_n(|\Omega_{v_n}|-a)^+$).
We now have 4 steps that are very similar to the 4 steps used in the previous part to show
that $\mu_+(h_n)$ is finite. \\

\textbf{Step 1: Euler-Lagrange equation for $v_n$.}
If $v\in\F$ is such that $|\Omega_v|\le|\Omega_{v_n}|$, we have
$J(v_n)\le J(v)$. Then, as in Lemma \ref{euler} we can
write the Euler-Lagrange equation $(\ref{eulern})$ for $v_n$
in $B_R$ for some $\Lambda_n$. \\

\textbf{Step 2: $\Lambda_n\geq \alpha_n$.} Since $|\Omega_{v_n}|>a$ the proof
is the same as step 2 in the first part, with $\alpha_n$ instead of $\mu_n$. \\

\textbf{Step 3: $v_n$ strongly converge to some $v$.} As in step 3 above, we just
write,
\[J(v_n)+\alpha_n(|\Omega_{v_n}|-a)^+\le J(u),\]
to get (up to a sub-sequence) that $v_n$ weakly converges
in $H^1_0(D)$, strongly in $L^2(D)$ and almost-everywhere to $v\in\F$.
We have $a<|\Omega_{v_n}|\le a+h_n$ and so $\lim_{n\to\infty}|\Omega_{v_n}|=a$.
As in step 3 above, we deduce that $v$ is a solution of $(\ref{probfonc2})$, and
using
\[J(v_n)+\alpha_n(|\Omega_{v_n}|-a)\le J(v), \]
we get the strong convergence in $H^1_0(D)$. \\

\textbf{Step 4: $\lim\Lambda_n=\Lambda$.}
The proof is the same as in step 4 of the first part of the proof. We write the Euler-Lagrange
equation for $v$ in $D$ and use $u=v$ outside $B_R$. We get that
$\lim\Lambda_n=\Lambda$ by letting $n$ go to $+\infty$ in the
Euler-Lagrange equation for $v_n$ in $B_R$ (using the strong convergence of $v_n$).
\\

We can now conclude this second part: if $\Lambda>0$, we have, for $n$ large enough,
\[\mu_+(h_n)-\eps=\alpha_n\le\Lambda_n\le\Lambda + \eps, \]
and so $\mu_+(h_n)\le\Lambda+2\eps$.

If $\Lambda=0$ we have
\[\mu_+(h_n)/2=\alpha_n\le\Lambda_n\le\eps, \]
and so $0\le\mu_+(h_n)\le2\eps$.\\
In both cases, we have $\Lambda\leq\mu_+(h_n)\leq\Lambda+2\eps.$\\
~\\\textbf{Third part: $\lim\mu_-(h)=\Lambda$. } \\
Let $h_n$ be a sequence decreasing to 0, and let $\eps>0$. Because $h\to\mu_-(h)$ is
increasing, we just have to show that $\lim_{n\to\infty}\mu_-(h_n)\geq\Lambda-\eps$ for
a sub-sequence of $h_n$.

We first see that $\mu_-(h)\le\Lambda$.
Let $\Phi\in C^{\infty}_0(B_R,\R^d)$ be such that
$\int_{B_R}\nabla\cdot \Phi=1$ and let $u_t=u(x+t\Phi(x))$ for $t\geq 0$. We have (using the
proof of Lemma \ref{euler}),
\[a-h\le |\Omega_{u_t}|=a-t+o(t)\le a, \]
\[J(u_t)=J(u)+t\Lambda +o(t). \]
Now, using $(\ref{mu})$, we have
\[J(u)+\mu_-(h)a\le J(u)+t\Lambda +\mu_-(h)(a-t)+o(t), \]
and we get $\mu_-(h)\le\Lambda$.

Let $v_n$ be a solution of the following minimization problem,
\begin{equation}\label{vnmu}
J(v_n)+(\mu_-(h_n)+\eps)\left(|\Omega_{v_n}|-(a-h_n)\right)^+
=\min_{w\in\F,\ |\Omega_w|\le a} \left\{J(w)+(\mu_-(h)+\eps)\left(|\Omega_{w}|-(a-h_n)\right)^+\right\}.
\end{equation}
We will first see that,
\[a-h_n\le |\Omega_{v_n}|< a. \]
If $|\Omega_{v_n}|=a$ we have,
\[J(u)+(\mu_-(h_n)+\eps)|\Omega_u|\le J(v_n)+(\mu_-(h_n)+\eps)|\Omega_{v_n}|
\le J(w) + (\mu_-(h_n)+\eps)|\Omega_w|, \]
for $w\in\F$ with $a-h_n\le |\Omega_w|\le a$ which contradicts the definition
of $\mu_-(h_n)$.\\
Now, if $|\Omega_{v_n}|<a-h_n$, we have
$J(v_n)\le J(v_n+t\varphi)$ for every $\varphi\in C^{\infty}_0(B_R)$ with
$|\{\varphi\neq 0\}|<a-h_n-|\Omega_{v_n}|$. And we get that
$-\Delta v_n=\lambda_a v_n$ in $B_R$ and so, we have $v_n\equiv 0$ on $B_R$
or $v_n>0$ on $B_R$, but this last case contradicts $|\Omega_{v_n}|<a$.
If $v_n\equiv 0$ on $B_R$, because $v_n=u$ outside $B_R$, we get $u\in H^1_0(B_R)$,
and using $J(u)\le J(v_n),$
\[\int_{B_R}|\nabla u|^2-\lambda_a\int_{B_R}u^2\le 0. \]
We now deduce ($u\not \equiv 0$ on $B_R$) that
$\lambda_a\geq \lambda_1(B_R)$, which is a contradiction, at least for $R$
small enough.

We now study the sequence $v_n$ in a very similar way than above.\\

\textbf{Step 1: Euler-Lagrange equation for $v_n$.}
$J(v_n)\le J(v)$ for $v\in\F$ with $|\Omega_v|\le|\Omega_{v_n}|$, so we have
an Euler-Lagrange equation $(\ref{eulern})$ for $v_n$ in $B_R$ for some
$\Lambda_n$. \\

\textbf{Step 2: $\Lambda_n\le (\mu_-(h_n)+\eps)$}. Since $|\Omega_{v_n}|<a$, we
take $\Phi\in C^{\infty}_0(B_R,\R^d)$ with $\int_{B_R}\nabla\cdot \Phi=-1$ and
$v_n^t(x)=v_n(x+t\Phi(x))$ for $t\geq 0$ small.
We have $|\Omega_{v^t_n}|=|\Omega_{v_n}|+t+o(t)\le a$ and
$J(v_n^t)=J(v_n)-\Lambda_n t+o(t)$ and writting $(\ref{vnmu})$ with
$w=v_n^t$ we get the result. \\

\textbf{Step 3: $v_n$ strongly converge to some $v$.} As in step 3 above we just
write that
\[J(v_n)+(\mu_-(h_n)+\eps)\left(|\Omega_{v_n}|-(a-h_n)\right)\le J(u)+(\mu_-(h_n)+\eps)h_n, \]
to get (up to a sub-sequence) that $v_n$ weakly converge
in $H^1_0(D)$, strongly in $L^2(D)$ and almost-everywhere to $v\in\F$.
We have $a-h_n<|\Omega_{v_n}|\le a$ and so $\lim_{n\to\infty}|\Omega_{v_n}|=a$.
As in step 3 above, we deduce that $v$ is a solution of $(\ref{probfonc2})$, and
using
\[J(v_n)+(\mu_-(h_n)+\eps)\left(|\Omega_{v_n}|-(a-h_n)\right)
\le J(v)+(\mu_-(h_n)+\eps)\left(|\Omega_{v}|-(a-h_n)\right)^+, \]
we get the strong convergence in $H^1_0(D)$. \\

\textbf{Step 4: $\lim\Lambda_n=\Lambda$.}
The proof is exactly the same as in step 4 above in the study of
the limit of $\mu_+(h_n)$.

Now we have, using steps 2 and 4, for $n$ large enough,
\[\Lambda-\eps\le\Lambda_n\le\mu_-(h_n)+ \eps\le \Lambda+\eps, \]
and so $\lim_{n\to\infty}\mu_-(h_n)=\Lambda$. \\
~\\
\textbf{Fourth part: $\Lambda>0$.} We would like to show that $\Lambda>0$ (which implies $\mu_-(h)>0$ for $h$ small enough).
We argue by contradiction and we suppose that $\Lambda=0$. The proof is very close to the proof of Proposition
6.1 in \cite{B}. We start with the following proposition:

\begin{prop}\label{majgrad} Assume $\Lambda=0$. Then, there exists $\eta$ a decreasing function with
$\lim_{r\to 0}\eta(r)=0$ such that, if $x_0\in B_{R/2}$ and
$B(x_0,r)\subset B_{R/2}$ with $|\{u=0\}\cap B(x_0,r)|>0$, then
\begin{equation}\label{eq:grad}
\frac{1}{r}\moyg_{\partial B(x_0,r)}u\le \eta(r).
\end{equation}
\end{prop}
\textbf{Proof of Proposition \ref{majgrad}. } Let $x_0,r$ be as above, and we set $B_r=B(x_0,r)$. Let $v$ be defined by,
\[\left\{\begin{array}{cccl}
-\Delta v & = &\lambda_a u &\textrm{ in } B_r\\
v & = & u &\textrm{ on } \partial B_r,
\end{array}
\right.
\]
and $v=u$ outside $B_r$. We have $v>0$ on $B_r$. We get, using
$(\ref{lambda})$,
\begin{equation}\label{moyp1}
\int_{B_r}(|\nabla u|^2-|\nabla v|^2)
-\lambda_a \int_{B_r}(u^2 -v^2)
\le   \mu_+(\omega_d r^d)|\{u=0\}\cap B_r| ,
\end{equation}
we also get (using $-\Delta v=\lambda_au$ in $B_r$),
\begin{eqnarray}\label{moyp2}
\int_{B_r}(|\nabla u|^2-|\nabla v|^2)
-\lambda_a\int_{B_r} u^2-v^2
& = &  \int_{B_r}\nabla (u-v).\nabla (u-v+2v)
-\lambda_a\int_{B_r} u^2-v^2 \nonumber \\
 & = &\int_{B_r}|\nabla (u-v)|^2+\lambda_a\int_{B_r}(u-v)^2.
\end{eqnarray}
Now, with the same computations as in \cite{AC},\cite{GSh} (with $\lambda_a u$
instead of $f$) we get,
\begin{equation}\label{moyp3}
|\{u=0\}\cap B_r|\left(\frac{1}{r}\moyg_{\partial B_r}u \right)^2
\le C \int_{B_r}|\nabla (u-v)|^2.
\end{equation}
Now, using $(\ref{moyp1}),(\ref{moyp2})$ and $(\ref{moyp3})$ we
get the result.\qed
~\\
\textbf{End of proof of Theorem \ref{theopena}. }
Now, the rest of the proof is the same as Proposition 6.2 in \cite{B} with
$\lambda_a u$ instead of $f\chi_{\Omega_u}$. The idea is that, from the estimate (\ref{eq:grad}) of Proposition \ref{majgrad}, $\nabla u$ tends
to 0 at the boundary, and consequently the measure $\Delta u$ does not charge the boundary $\partial\Omega_u$. It follows that
$-\Delta u=\lambda_a u$ in $B_R$, which, by strict maximum principle, contradicts that $u$ is zero on some part of $B_R$.\qed

\end{section}

\begin{section}{Proof of Theorem \ref{theoreg}}
Let $\Omega^*$ be a solution of (\ref{probforme}). Then $u=u_{\Omega^*}$ is a solution of (\ref{probfonc2}), and thus satisfies Proposition \ref{lip}
 and Theorem \ref{theopena};
 moreover, $\Omega^*=\Omega_u$. Like in the previous section, we work in $B$, a small ball centered in $\partial\Omega_u$.
Since the approach is local, we will show regularity for the part of
$\partial\Omega_u$ included in $B$; but $B$ can be centered on every point of
$\partial\Omega_u\cap D$, so this is of course enough to lead to the announced results in Theorem \ref{theoreg}.\\
Coupled with Remark \ref{remCaf}, we conclude that it is sufficient to prove:
\begin{equation}\label{acc}
\left.
\begin{array}{cll}
(a)& \Omega^*\textrm{ has finite perimeter in }B\textrm{ and }
\haus((\partial\Omega^*\setminus\partial^*\Omega^*)\cap B)=0\\[3mm]
(b)& \Delta u_{\Omega^*}+\lambda_1(\Omega^*)u_{\Omega^*}=\sqrt{\Lambda} \haus\lfloor\partial\Omega^*\;\;in\;\;B,&\\[3mm]
(c)& \textrm{if d=2, }\partial\Omega^*\cap B=\partial^*\Omega^*\cap B.
\end{array}
\right\}
\end{equation}
We use the same arguments as in \cite{AC} and \cite{GSh}, but we have to deal with the term
in $\int u^2$ instead of $\int fu$ (in $\cite{GSh}$).
So we first start with the following technical lemma.

\begin{lemma}\label{lem}
There exist $C_1,C_2,r_0>0$ such that,
for $B(x_0,r)\subset B$ with $r\le r_0$,
\begin{equation}\label{moyencadre}
\begin{array}{ccl}
\textrm{if }\;\frac{1}{r}\moyg_{\partial B(x_0,r)}u \geq C_1
& \textrm { then } & u>0 \textrm{ on } B(x_0,r), \\
\textrm{if }\;\frac{1}{r}\moyg_{\partial B(x_0,r)}u \le C_2
& \textrm { then } & u \equiv 0 \textrm{ on } B(x_0,r/2).
\end{array}
\end{equation}
\end{lemma}
\begin{proof} The first point comes directly from the proof
of Proposition \ref{majgrad}.
We take the same $v$ and,
using equation $(\ref{moyp3})$, we see that there exists $C_1$ such
that if $\frac{1}{r}\moy_{\partial B(x_0,r)}u \geq C_1$, then
$|\{u=0\}\cap B(x_0,r)|=0$.

For the second part we argue as in Theorem 3.1 in \cite{ACF}. We will denote
$B_r$ for $B(x_0,r)$. In this proof, $C$ denotes (different)
constants which depend only on $a,d,D,u$ and $B$, but not
on $x_0$ or $r$.

Let $\eps>0$ small and such that $\{u=\eps\}$ is smooth (true for almost every $\eps$), let
$D_{\eps}=(B_r\setminus \overline{B}_{r/2})\cap\{u>\eps\}$
and $v_{\eps}$ be defined by
$$\left\{\begin{array}{cccl}
-\Delta v_{\eps} & = & \lambda_a u& \textrm { in } D_{\eps}\\
v_{\eps} & = &u&  \textrm{ in } D\setminus B_r \\
v_{\eps} & = &u &  \textrm{ in } B_r\cap\{u\le \eps \} \\
v_{\eps} & = & \eps &\textrm{ in } \overline{B}_{r/2} \cap\{u> \eps \}.
\end{array}
\right.$$
We see that $u-v_{\eps}$ is harmonic in $D_{\eps}$.

We now show that $(v_{\eps}-u)_{\eps}$ is bounded in $H^1(D)$, for small $\eps>0$.
Let $\varphi$ be in $C^{\infty}_0(B_r)$ with $0\le\varphi\le 1$ and
$\varphi\equiv 1$ on $\overline{B}_{r/2}$.  Let $\Psi=(1-\varphi)u+\eps\varphi=
u + \varphi(\eps-u)$. We have:
\[\Psi-u=0=v_{\eps}-u \textrm{ on } \partial B_r\cup(\partial D_{\eps}\cap
(B_r\setminus \overline{B}_{r/2})), \]
and
\[\Psi-u=\eps-u=v_{\eps}-u \geq -\|u\|_{\infty}
\textrm{ on } \partial D_{\eps}\cap \partial B_{r/2}, \]
so using that $v_{\eps}-u$ is  harmonic, we get
$-\|u\|_{\infty}\le v_{\eps}-u\le 0$ on $D_{\eps}$
and,
\[\int_{D_{\eps}}|\nabla (v_{\eps}-u)|^2 \le
\int_{D_{\eps}}|\nabla (\Psi-u)|^2. \]
Now, using that $\nabla \Psi=\nabla u(1-\varphi)-(\nabla\varphi) u + \eps\nabla\varphi $
and the $L^{\infty}$ bounds for $u$ and $\nabla u$, we see
that $v_{\eps}-u$ is bounded in $H^1(D)$.

Now, up to a subsequence, $v_{\eps}$ weakly converges in $H_0^1(D)$
to $v$ such that:
$$\left\{\begin{array}{cccl}
-\Delta v & = & \lambda_a u &\textrm { in } (B_r\setminus\overline{B}_{r/2})\cap\Omega_u\\
v & = &u  &\textrm{ in } D\setminus B_r \\
v & = & 0 &\textrm{ in } \overline{B}_{r/2} \cup (B_r\cap\{u=0\}).
\end{array}
\right.$$
Using $(\ref{mu})$  with $h=|B_{r/2}|$, and $u=v$ in $D\setminus B_r$, we have:
\[
\int_{B_r}|\nabla u|^2-\lambda_a\int_{B_r}u^2 + \mu_-(h) |\Omega_u\cap B_r|
\le \int_{B_r}|\nabla v|^2-\lambda_a\int_{B_r}v^2
+ \mu_-(h) |\Omega_v\cap B_r|,
\]
and so,
\begin{eqnarray}
&&\hspace{-1 cm}\int_{B_{r/2}}|\nabla u|^2\mu_-(h) |\Omega_u\cap B_{r/2}|
\nonumber \\
&\le& \int_{B_r\setminus B_{r/2}}\nabla(v-u).\nabla (u-v+2v)
-\lambda_a\int_{B_r\setminus B_{r/2}}(v^2-u^2)
+\lambda_a\int_{B_{r/2}} u^2 \nonumber \\
&\le&  \liminf_{\eps\to\ 0 }
2\int_{D_\eps}\nabla(v_{\eps}-u).\nabla v_{\eps}
-\lambda_a\int_{D_{\eps}}(v_{\eps}^2-u^2)
+\lambda_a\int_{B_{r/2}} u^2 \nonumber \\
&=&\liminf_{\eps\to\ 0 }  2\int_{\partial B_{r/2}\cap\{u>\eps\}}
(\eps-u)\frac{\partial v_{\eps}}{\partial n}
+2\lambda_a\int_{D_{\eps}}(v_{\eps}-u)u
-\lambda_a\int_{D_{\eps}}(v_{\eps}^2-u^2)
+\lambda_a\int_{B_{r/2}} u^2 \nonumber \\
&=&\liminf_{\eps\to\ 0 } 2\int_{\partial B_{r/2}\cap\{u>\eps\}}
(\eps-u)\frac{\partial v_{\eps}}{\partial n}
+\lambda_a\int_{D_{\eps}}(2uv_{\eps}-u^2-v_{\eps}^2)
+\lambda_a\int_{B_{r/2}} u^2
\nonumber \\
&\le&\liminf_{\eps\to\ 0 }
 2\int_{\partial B_{r/2}\cap\{u>\eps\}}
(\eps-u)\nabla v_{\eps}.\overrightarrow{n}
+\lambda_a\int_{B_{r/2}} u^2,
\label{mino1}
\end{eqnarray}
where $\overrightarrow{n}$ is the outward normal of $D_{\eps}$
 and so the inward normal of $B_{r/2}$.
Let $w_{\eps}$ be such that,
$$\left\{\begin{array}{cccl}
-\Delta w_{\eps} & = &
\lambda_a u  &\textrm{ on } B_r\setminus \overline{B}_{r/2}\\
w_{\eps} & = & u &\textrm{ on } \partial B_r\cap\{u>\eps\} \\
w_{\eps} & = & \eps &\textrm{ on }
\left(\partial B_r\cap\{u\le\eps\}\right)\cup\partial B_{r/2}. \\
\end{array}
\right.$$
Because $w_{\eps}\geq \eps$ on $\partial(B_r\setminus B_{r/2})$
and super-harmonic in $B_r\setminus\overline{B}_{r/2}$, we get that
$w_{\eps}\geq\eps$ in $B_r\setminus \overline{B}_{r/2}$.
In particular $w_{\eps}\geq v_{\eps}=\eps$ in $\partial D_{\eps}\cap
(B_r\setminus\overline{B}_{r/2})$. Moreover, we also have $w_{\eps}\geq v_{\eps}$ on
$\partial D_{\eps}\cap(\partial B_r\cup\partial B_{r/2})$, and since $w_{\eps}-v_{\eps}$
is harmonic in $D_{\eps}$, we get $w_{\eps}\geq v_{\eps}$ in
$D_{\eps}$. Using $w_{\eps}=v_{\eps}=\eps$ on
$\partial B_{r/2}\cap \{u>\eps\}$, we can now compare the
gradients of $w_{\eps}$ and
$v_{\eps}$ on this set,
\begin{equation}\label{mino2}
0\le -\nabla v_{\eps}.\overrightarrow{n}
\le -\nabla w_{\eps}.\overrightarrow{n}\textrm{ on }\partial B_{r/2}\cap \{u>\eps\}.
\end{equation}

Let now $w^0_{\eps}$ be defined by $w^0_{\eps}=w_{\eps}$ on
$\partial(B_r\setminus \overline{B}_{r/2})$ and harmonic in
$B_r\setminus B_{r/2}$. We use now the following estimate:
\begin{equation}\label{mino3}
0\le -\nabla w^0_{\eps}.\overrightarrow{n}\le
\frac{C}{r}\moyg_{\partial B_r}(u-\eps)^+\le C\gamma
\textrm{ on } \partial B_{r/2},
\end{equation}
where $\gamma=\frac{1}{r}\moy_{\partial B_r}u$ (to get this estimate, we can first prove, using a comparison argument, that $|\nabla w^0_\eps|\leq\frac{C}{r}\|w^0_\eps-\eps\|_{\infty,B_{3r/4}\setminus B_{r/2}}$, and then conclude using again maximum principle and Poisson formula for functions that are harmonic in a ball).
Let $w^1_{\eps}=w_{\eps}-w^0_{\eps}$, we have
$w^1_{\eps}=0$ on $\partial (B_r\setminus B_{r/2})$ and
$-\Delta w^1_{\eps}=\lambda_a u$ in $B_r\setminus \overline{B}_{r/2}$ and so,
\begin{equation}\label{mino4}
\|\nabla w^1_{\eps}\|_{\infty,B_r\setminus B_{r/2}}
\le  Cr \|u\|_{\infty}\le Cr.
\end{equation}
Now using $(\ref{mino1}),(\ref{mino2}),(\ref{mino3})$ and
$(\ref{mino4})$ we get,
\begin{equation}\label{mino5}
L:=\int_{B_{r/2}}|\nabla u|^2+ \mu_-(h) |\Omega_u\cap B_{r/2}|
\le C(\gamma+ r)\int_{\partial B_{r/2}}u
+\lambda_a\int_{B_{r/2}}u^2.
\end{equation}
Our goal is now to bound from above the right-hand of this inequation
with $CL(\gamma+r)$: and so if $\gamma$ and $r$ are small enough we
will get $L=0$ and so $u\equiv 0$ in $B_{r/2}$.

We now give an estimate of $\|u\|_{\infty,B_{r/2}}$ in
term of $\gamma$. Let $w=0$ on $\partial B_r$ and
$-\Delta w=\lambda_a u$ in $B_r$. We have
(using $(\ref{deltaumes})$)
$\Delta(u-w)=\Delta u + \lambda_a u\geq 0$ in
$B_r$ and $u-w=u$ on $\partial B_r$ so,
\[\|u-w\|_{\infty,B_{r/2}}\le C\moyg_{\partial B_r} u\le C\gamma r. \]
We also have that
\[\|w\|_{\infty,B_r}\le Cr^2\|u\|_{\infty,B_r}\le Cr^2, \]
and finally,
\begin{equation}\label{mino6}
\|u\|_{\infty,B_{r/2}}\le C(\gamma r + r^2).
\end{equation}
We now write (using $(\ref{mino6})$),
\begin{eqnarray}
\int_{\partial B_{r/2}}u & \le &
C \left( \int_{B_{r/2}}|\nabla u| + \frac{1}{r}\int_{B_{r/2}}u
\right)\nonumber \\
& \le &
C\left(\frac{1}{2} \int_{B_{r/2}}|\nabla u|^2 +
\frac{1}{2}|\Omega_u\cap B_{r/2}|
+ \frac{1}{r}|\Omega_u\cap B_{r/2}|\|u\|_{\infty,B_{r/2}}\right).
\nonumber
\end{eqnarray}
Here we use Theorem \ref{theopena} to see that there
exists $h_0$ such that
\[\frac{\Lambda}{2}\le \mu_-(h)\le \Lambda,
\ 0< h\le h_0. \]
And so, we have
\begin{eqnarray}
\int_{\partial B_{r/2}}u  & \le &
C \left( \int_{B_{r/2}}|\nabla u|^2 + \mu_-(h)|\Omega_u\cap B_{r/2}|
+ C|\Omega_u\cap B_{r/2}|(\gamma +r) \right) \nonumber
\\
&\leq&CL(1+\gamma+r) \label{mino7},
\end{eqnarray}
with $C$ independent of $r$ for every $r$ small enough
such that $h=|B_{r/2}|\le h_0$.
We also have (using $(\ref{mino6})$)
\begin{equation}\label{mino8}
\int_{B_{r/2}}u^2\le C |\Omega_u\cap B_{r/2}|(\gamma r + r^2)
\le CL(\gamma r+r^2).
\end{equation}
We now get, from $(\ref{mino5}),(\ref{mino7})$
and $(\ref{mino8})$, if $\gamma\le 1$ and $r\le 1$,
\[L\le C(\gamma + r)L(1+\gamma+r)+CL(\gamma r+r^2)
\le CL(\gamma + r),\]
and, if we suppose $r\le \frac{1}{2C}$ we get,
\[L\le CL\gamma +\frac{L}{2},\]
and so, if $\gamma<\frac{1}{2C}$ we get $L=0$ and
$u\equiv 0$ on $B_{r/2}$. \qed
\end{proof}
\bigskip
With the help of this lemma, we are now able to successively prove the three properties (a),(b) and (c) of (\ref{acc}).\\
~\\
\textbf{Proof of (a). }The proof is now, using $(\ref{moyencadre})$ in lemma \ref{lem},
the same as in \cite{GSh} or
in \cite{AC}. Here are the main steps:
we first show that there exists $C_1,C_2$ and
$r_0$ such that, for every $B(x_0,r)\subset B$
with $r\le r_0$,
\[0<C_1\le\frac{|B(x_0,r)\cap\Omega_u|}{|B(x_0,r)|}
\le C_2<1, \]
and
\[C_1r^{d-1}\le (\Delta u+\lambda_a u)
(B(x_0,r))\le C_2 r^{d-1}. \]
The proof is the same as in \cite{GSh}
with $\lambda_a u$ instead of $f$.
It gives directly (using the Geometrical
measure theory, see section 5.8 in \cite{EG}) the first point of Theorem \ref{theoreg}.\\~\\
\textbf{Proof of (b). }For the second point, we see that $\Delta u +\lambda_a u$
is absolutely continuous with respect to
$\haus\lfloor\partial\Omega_u$ which is
a Radon-Measure (using the first point), so we
can use Radon's Theorem. To compute the
Radon's derivative, we argue as in Theorem 2.13
in \cite{GSh} or (4.7,5.5) in \cite{AC}. The
main difference is that here, we
have to use $(\ref{limmu})$ in Theorem \ref{theopena}
to show that, if $u_0$ denotes a blow-up limit of $u(x_0+rx)/r$
(when $r$ goes to 0), then $u_0$ is such that,
\[\int_{B(0,1)}|\nabla u_0|^2 + \Lambda
|\{u_0\neq 0\}\cap B(0,1)| \le
\int_{B(0,1)}|\nabla v|^2 + \Lambda
|\{v\neq 0\}\cap B(0,1)|, \]
for every $v$ such that $v=u_0$ outside
$B(0,1)$. To show this, in \cite{AC} or
in \cite{GSh} the authors use only perturbations in
$B(x_0,r)$ with $r$ goes to $0$, so using
$(\ref{limmu})$, we get the same result.
We can compute the Radon's derivative and
get (in $B$)
\[\Delta u+\lambda_a u =\sqrt{\Lambda}\haus\lfloor\partial
\Omega_u.\]
Now, $u$ is a weak-solution in the sense of \cite{GSh}
and \cite{AC} and we directly get the analytic regularity
of $\partial^*\Omega_u$ (this regularity
is shown for weak-solutions).\\~\\
\textbf{Proof of (c). }If $d=2$, in order to have
the regularity of the whole boundary, we
have to show that Theorem 6.6 and Corollary 6.7
in \cite{AC} (which are for solutions and not weak-solutions)
are still true for our problem.  The Corollary directly
comes from the Theorem.
So we need to show that, if
$d=2$ and $x_0\in\partial\Omega_u$, then
\begin{equation}\label{dim2}
\lim_{r\to 0}\moyg_{B(x_0,r)}
\max\{\Lambda-|\nabla u|^2,0\}=0.
\end{equation}
We argue as in Theorem 6.6 in \cite{AC}. Let
$\zeta\in C^{\infty}_0(B)$ be nonnegative and let
$v=\max\{u-\eps\zeta,0\}$. Using $(\ref{mu})$ with
this $v$ and $h=|0<u\le\eps\zeta|\le |\{\zeta\neq 0\}|$ we get,
\begin{eqnarray*}
\mu_-(h)|0<u\le\eps\zeta| & \le &
\int |\nabla v|^2 - \int |\nabla u|^2
+\lambda_a\int (u^2-v^2) \\
& = &
\int |\nabla\min\{\eps\zeta,u\}|^2
- 2\int \nabla u .\nabla\min\{\eps\zeta,u\}
\\
& & +
\lambda_a\int_{\{u<\eps\zeta\}}u^2
-\lambda_a\int_{\{u\geq\eps\zeta\}}(\eps\zeta)^2
+2\lambda_a \int_{\{u\geq\eps\zeta\}}u\eps\zeta.
\end{eqnarray*}
Using $-\Delta u=\lambda_a u$ in $\Omega_u$ we get:
\[\int \nabla u .\nabla\min\{\eps\zeta,u\}=
\lambda_a \int u\min\{\eps\zeta,u\}=
\lambda_a\int_{\{u<\eps\zeta\}}u^2
+\lambda_a\int_{\{u\geq \eps\zeta\}}u\eps\zeta, \]
and so,
\[\mu_-(h)|0<u\le\eps\zeta|
\le \int_{\{u<\eps\zeta\}}|\nabla u|^2
+\int_{\{u\geq \eps\zeta\}}\eps^2|\nabla\zeta|^2
-\lambda_a\int_{\{u<\eps\zeta\}}u^2 -
\lambda_a\int_{\{u\geq\eps\zeta\}}(\eps\zeta)^2,
\]
and so, we can deduce that,
\[\int_{\{0<u<\eps\zeta\}}(\Lambda-|\nabla u|^2)
\le \int_{\{u\geq \eps\zeta\}}\eps^2|\nabla\zeta|^2
+(\Lambda-\mu_-(h))h. \]
The only difference now with \cite{AC}
is the last term. Using
Theorem \ref{theopena}, we see that
$(\Lambda-\mu_-(h))h=o(h)$, so we can
choose the same kind of $\zeta$ and $\eps$ as in \cite{AC}
to get $(\ref{dim2})$ (see Theorem 5.7
in \cite{B} for more details).\qed
\end{section}

\begin{section}{Appendix}
In this appendix, we discuss the hypothesis ``$D$ is connected''. We begin with the following example,
taken from \cite{BHP}.
\begin{example}(from \cite{BHP})
We take $D=D_1\cup D_2$, where $D_1,D_2$ are disjoint disks in $\R^2$ of radius $R_1,R_2$ with $R_1>R_2$.
If $a=\pi R_1^2+\eps$, then the solution $u$ of $(\ref{probfonc})$ coincides with the first eigenfunction of $D_1$
and is identically 0 on $D_2$, and thus $\Omega_u=D_1$ and $|\Omega_u|<a$.\\
In this case, we can choose an open subset $\omega$ of $D_2$ with $|\omega|=\eps$. Then $\Omega^*:=D_1\cup\omega$
is a solution of $(\ref{probforme})$. Since $\omega$ may be chosen as irregular as one wants, this proves that optimal domains are not regular in
 general.
\end{example}
However, we are able to prove the following proposition.
\begin{prop}[The non-connected case]\label{noncon}
If we suppose that $D$ is not connected, the problem $(\ref{probfonc})$
still has a solution $u$ which is locally Lipschitz
continuous in $D$. If $\omega$ is
any open connected component of $D$, we have three cases:
\begin{enumerate}
\item either $u>0$ on $\omega$,
\item or $u=0$ on $\omega$,
\item or $0<|\Omega_u\cap\omega|<|\omega|$, and $\partial\Omega_u$ has the same regularity as stated in Theorem \ref{theoreg}.
\end{enumerate}
If $|\Omega_u|<a$, then only the first two cases can appear.
\end{prop}
\begin{remark} It follows from Proposition \ref{noncon} that we obtain the same regularity as in the connected case. Indeed, in the first two cases,
$\partial\Omega^*\cap\omega=\partial\Omega_u\cap\omega=\emptyset$.
\end{remark}
\begin{remark}
 To summarize, in all cases, there exists a solution $\Omega^*$ to (\ref{probforme}) which is regular in the sense of Theorem \ref{theoreg}, but there may be some other non regular optimal shape. And if $D$ is connected, any optimal shape is regular.
\end{remark}

\begin{proof} The existence and the Lipschitz regularity are stated in Proposition \ref{lip}.\\
If $u=0$ a.e. on $\omega$, then we get $u=0$ on $\omega$
by continuity.
\\If $u>0$ a.e. on $\omega$, by Lemma \ref{mes}, $u>0$ everywhere
in $\omega$.\\
If $0<|\Omega_u\cap\omega|<|\omega|$, the restriction of $u$ to
$\omega$ is of course solution of $(\ref{probfonc2})$ with $\omega$
instead of $D$ and $|\omega\cap\Omega_u|$ instead of $a$. We then may apply Theorem \ref{theoreg}.\\
Finally, if $|\Omega_u|<a$, we may write $J(u)\le J(u+t\varphi)$ for all $t\in(-\eps,\eps)$ and for
all $\varphi\in C^{\infty}_0(D)$ such that
$|\Omega_\varphi|<a-|\Omega_u|$ and so:
\[0=\left. \frac{{\rm d}J(u+t\varphi)}{{\rm d}t} \right|_{t=0}
=2\int_D (\nabla u.\nabla\varphi)-2\lambda_a \int_D u\varphi. \]
That is $-\Delta u=\lambda_a u$ in $D$ and the third
case is not possible since by maximum principle $u>0$ or $u=0$ on each connected component of $D$.\qed
\end{proof}

\end{section}

\bibliographystyle{plain}


\end{document}